\documentclass[11pt,a4paper]{amsart}
\usepackage{amsthm,amsmath,amssymb,latexsym}
\usepackage{array,booktabs,ragged2e}
\usepackage{graphicx}
\usepackage[margin=0.7cm]{caption}
\usepackage{color}

\newtheorem{theorem}{Theorem}[section]
\newtheorem{corollary}[theorem]{Corollary}

\newtheorem{lemma}[theorem]{Lemma}

\newtheorem*{main}{Main Theorem}

\theoremstyle{definition}
\newtheorem{definition}[theorem]{Definition}

\DeclareMathOperator{\hDim}{hDim}
\DeclareMathOperator{\Std}{Std}
\DeclareMathOperator{\Spec}{Spec}

\newcommand{\F}{\mathbb{F}}

\newcommand{\N}{\mathbb{N}}

\newcounter{thmlistcnt}
	{\setcounter{thmlistcnt}{0}%
	\begin{list}{\emph{(\roman{thmlistcnt})}}{%
		\usecounter{thmlistcnt}%
		\setlength{\topsep}{0pt}%
		\setlength{\leftmargin}{0pt}%
		\setlength{\itemsep}{0pt}%
		\setlength{\itemindent}{17pt}}%
	}%
	{\end{list}}%

\title{Hausdorff dimension in $R$-analytic profinite groups}

\author[Fern\'{a}ndez-Alcober]{Gustavo A. Fern\'{a}ndez-Alcober}
\address{Department of Mathematics\\ University of the Basque Country UPV/EHU\\
48080 Bilbao, Spain}
\email{gustavo.fernandez@ehu.eus}

\author[Giannelli]{Eugenio Giannelli}
\address[]{Department of Mathematics, University of Kaiserslautern,
P.O. Box 3049, 67655 Kaiserslautern, Germany}
\email{gianelli@mathematik.uni-kl.de}

\author[Gonz\'{a}lez-S\'{a}nchez]{Jon Gonz\'{a}lez-S\'{a}nchez}
\address{Department of Mathematics\\ University of the Basque Country UPV/EHU\\
48080 Bilbao, Spain}
\email{jon.gonzalez@ehu.es}

\thanks{The first and third authors are supported by the Spanish Government, grants
MTM2011-28229-C02 and MTM2014-53810-C2-2-P, and by the Basque Government, grants IT753-13 and IT974-16.
The second author gratefully acknowledges financial support by the GRECA research group and by the ERC Advanced Grant 291512.}

\keywords{Profinite groups, analytic groups, Hausdorff dimension\vspace{3pt}}


\begin{document}

\begin{abstract}
We study the Hausdorff dimension of $R$-analytic subgroups in an $R$-analytic profinite group, where $R$ is a pro-$p$ ring whose associated graded ring is an integral domain.
In particular, we prove that the set of such Hausdorff dimensions is a finite subset of the rational numbers.
\end{abstract}

\maketitle

\section{Introduction}

The study of Hausdorff dimension in profinite groups was initiated by Abercrombie in \cite{Abercrombie}, and has attracted special attention in recent times; see, for example, \cite{Abert,BK,YA,Ershov,FZ,Siegenthaler}.
If $G$ is a countably based infinite profinite group, a \emph{(normal) filtration\/} of $G$ is a descending series
$\{G_n\}_{n\in\mathbb{N}}$ of (normal) open subgroups of $G$ which is a base of neighbourhoods of $1$.
We can define a metric on $G$ by
\[
d(x,y)
=
\inf \, \{|G:G_n|^{-1}\ |\ xy^{-1}\in G_n\},
\]
and the topology induced by $d$ coincides with the original topology in $G$.
The metric $d$ defines a Hausdorff dimension function on all subsets of $G$, which we denote by  $\hDim_{\{G_n\}}$.
As shown in \cite[Theorem 2.4]{YA}, if the filtration $\{G_n\}_{n\in\mathbb{N}}$ is normal and $H$ is a closed subgroup of $G$, then
\[
\hDim_{\{G_n\}}(H)
=
\liminf_{n\rightarrow\infty} \, \frac{\log |H:H\cap G_n|}{\log |G:G_n|}.
\]

In \cite{YA}, Barnea and Shalev studied the Hausdorff dimension in $p$-adic analytic pro-$p$ groups with respect to the filtration $\{G^{p^n}\}_{n\in\N}$,
and proved that
\[
\hDim_{\{G^{p^n}\}}(H)=\frac{\dim(H)}{\dim(G)}
\]
for every closed subgroup $H$ of $G$.
Here, if $G$ is an analytic group, $\dim(G)$ denotes the dimension of $G$ as an analytic manifold. 

In this paper we study Hausdorff dimension in $R$-analytic profinite groups, where $R$ is a pro-$p$ ring whose associated graded ring is an integral domain.
This is the natural setting for working with analytic groups over pro-$p$ rings, see \cite[Chapter 13]{Dixon}.
If $G$ is an $R$-analytic profinite group, then $G$ is virtually a pro-$p$ group, but the subgroups $G^{p^n}$ need not to form a filtration, since they need not to be open in $G$.
We consider instead the natural filtration induced by a standard open subgroup of $G$.
As we will see, the Hausdorff dimension of a closed subgroup $H$ of $G$ is the same with respect to the natural filtrations induced by all open standard subgroups of $G$.
This allows us to define the concept of \textit{standard Hausdorff dimension} of $H$, denoted by $\hDim_{\mathrm{Std}}(H)$.
We get the following theorem, in the spirit of the result of Barnea and Shalev. 

\begin{main}
Let $G$ be an $R$-analytic profinite group and let $H$ be an $R$-analytic subgroup of
$G$.
Then 
\begin{equation}
\label{T1}
\hDim_{\Std}(H)=\frac{\dim(H)}{\dim(G)}.
\end{equation}
\end{main}

An $R$-analytic subgroup $H$ of $G$ is a subgroup that is also an $R$-analytic submanifold of $G$; then $H$ is closed in $G$.
A crucial remark is that while the converse is true in a $p$-adic analytic group, i.e.\
every closed subgroup is analytic, it need not hold for an arbitrary pro-$p$ ring $R$.
For example, $\F_p[[t^d]]$ is a closed subgroup of $\F_p[[t]]$ for every positive integer $d$, but it is not analytic for $d>1$, even if it is an
$\F_p[[t]]$-analytic group in its own right.

Observe that (\ref{T1}) implies that the \textit{$R$-analytic spectrum\/}  of $G$, defined by
\[
\Spec_R(G)
=
\{ \hDim_{\Std}(H) \mid \text{$H$ is an $R$-analytic subgroup of $G$} \},
\]
is finite and consists of rational values.
On the other hand, the spectrum of $\F_p[[t]]$ corresponding to all closed subgroups is the full interval $[0,1]$, as shown in \cite[Lemma 4.1]{YA}.
Thus our main theorem is pointing to the fact that most closed subgroups of $\F_p[[t]]$ are non-analytic.

\vspace{10pt}

\noindent
\textit{Notation.\/}
We write $X^{(d)}$ to denote the cartesian product of $d$ copies of a set $X$.
The symbol $\subseteq_o$ indicates that a subset of a topological space is open.

\section{Preliminaries}

Throughout this paper $R$ is a pro-$p$ ring whose associated graded ring is an integral domain.
Hence $R$ is an integral domain, and we write $K$ for its field of fractions.
If $\mathfrak{m}$ is the maximal ideal of $R$, then $R/\mathfrak{m}$ is a finite field of characteristic $p$.
Set $H(n)=\dim_{R/\mathfrak{m}} (\mathfrak{m}^n/\mathfrak{m}^{n+1})$ for all $n\in\N$.
For large enough $n$, $H(n)$ coincides with a polynomial $P(n)$, called the \emph{Hilbert polynomial\/} of $R$ \cite[Chapter 8, Theorem C]{Eisenbud}.

Let $G$ be an $R$-analytic group.
Without loss of generality, we assume that the manifold structure is given by a full atlas.
A subgroup $H$ of $G$ is an \emph{analytic subgroup\/} if it is also an analytic submanifold of $G$.
We adopt Serre's definition of submanifold \cite[Section 3.11]{Serre}.
Every analytic subgroup is closed in $G$.
The converse is not true in general, but every open subgroup $H$ of $G$ (actually, any open subset) is a submanifold when considered with the restrictions of the charts of $G$, and $\dim H=\dim G$.

An $R$-analytic group $S$ with a global chart given by a homeomorphism $\phi:S\rightarrow (\mathfrak{m}^N)^{(d)}$ is called an \emph{$R$-standard group\/} of level $N$ if $\phi(1)=0$ and, for all $j\in\{1,\ldots,d\}$, there exist $F_j(X,Y)\in R[[X,Y]]$ (where $X$ and $Y$ are $d$-tuples of indeterminates), without constant term, such that
\begin{equation}
\label{formal group law}
\phi(xy)=F(\phi(x),\phi(y)), \quad \text{for every $x,y\in S$.}
\end{equation}
Here, $F=(F_1,\ldots,F_d)$ is called the \textit{formal group law\/} associated to $S$.
In these circumstances, we also say that $(S,\phi)$ is a standard group.

By using (\ref{formal group law}), $F$ defines a new group structure on $(\mathfrak{m}^N)^{(d)}$, other than its natural additive structure, and then
$\phi$ is an isomorphism between $S$ and $((\mathfrak{m}^N)^{(d)},F)$.
Every $R$-analytic group contains an open (and so analytic) $R$-standard subgroup $S$ \cite[Theorem 13.20]{Dixon}.
As a consequence, profinite $R$-analytic groups are countably based.

The formal group law of a standard group satisfies the following result \cite[pp.\ 331-334]{Dixon}. 

\begin{lemma}
\label{fgl}
Let $F$ be a formal group law of dimension $d$ associated to a standard group $S$.
Then
\[
F(X,Y)=X+Y+G(X,Y),
\]
where every monomial involved in $G$ has total degree at least $2$, and contains a non-zero power of $X_r$ and of $Y_s$ for some $r,s\in\{1,\ldots,d\}$.
Moreover,
\[
F(X,I(X))=0=F(I(X),X)
\]
for some $I(X)=-X+H(X)\in R[[X]]$, and every monomial involved in $H$ has total degree at least $2$.
\end{lemma}

It follows that $(\mathfrak{m}^{N+n})^{(d)}$ is a subgroup of $((\mathfrak{m}^N)^{(d)},F)$ for all $n\ge 0$.
This allows us to introduce a special type of filtrations in an $R$-analytic group.

\begin{definition}
Let $G$ be an $R$-analytic group and let $S$ be a standard open subgroup of $G$, with global chart $(S,\phi)$.
Then for every positive integer $n\ge 0$, and for every subset $A\subseteq S$, we define
\[
\mathfrak{m}^n_\phi A = \phi^{-1}(\mathfrak{m}^n\phi(A)).
\]
We say that $\{\mathfrak{m}^n_\phi S\}_{n\in\mathbb{N}}$ is the \emph{natural filtration\/} of $G$ induced by $S$.
\end{definition}

Observe that 
\[
\mathfrak{m}^n_\phi S=\phi^{-1}((\mathfrak{m}^{N+n})^{(d)}),
\]
which implies that $\mathfrak{m}^n_\phi S$ is an open subgroup of $G$, and also that
$\{\mathfrak{m}^n_\phi S\}_{n\in\mathbb{N}}$ is a filtration of $G$.
Actually, we have $\mathfrak{m}^n_\phi S\trianglelefteq S$
(see \cite[Proposition 13.22]{Dixon}).

\begin{lemma}
\label{iso}
Let $(S,\phi)$ be an $R$-standard group.
Then:
\begin{enumerate}
\item
For all $x,y\in S$ and $n\ge N$, we have $\phi(xy^{-1}) \in (\mathfrak{m}^n)^{(d)}$ if and only if
$\phi(x)-\phi(y) \in (\mathfrak{m}^n)^{(d)}$.
\item
For every $n\ge 0$, we have $|S:\mathfrak{m}^n_\phi S|=|(\mathfrak{m}^N)^{(d)}:(\mathfrak{m}^{N+n})^{(d)}|=q^{d\hspace{.4pt}f(n)}$,
where $f(n)=\sum_{i=N}^{N+n-1} \, H(i)$.
\item
$\phi$ is an isometry between the group $S$ with the metric induced by the filtration $\{\mathfrak{m}^n_\phi S\}_{n\in\mathbb{N}}$ and the group
$((\mathfrak{m}^N)^{(d)},+)$ with the metric induced by the filtration $\{(\mathfrak{m}^{N+n})^{(d)}\}_{n\in\mathbb{N}}$.
\end{enumerate}
\end{lemma}

\begin{proof}
(i)
Since $\phi(1)=0$, we may assume that $x\ne y$.
Let $k\in\N$ be such that
$\phi(xy^{-1})\in  (\mathfrak{m}^k)^{(d)} \smallsetminus (\mathfrak{m}^{k+1})^{(d)}$.
Then since
\[
\phi(x)
=
\phi(xy^{-1}y)
=
F(\phi(xy^{-1}),\phi(y))
=
\phi(xy^{-1}) + \phi(y) + G(\phi(xy^{-1}),\phi(y)),
\]
by Lemma \ref{fgl}, we have
\[
\phi(x) - \phi(y)
\equiv
\phi(xy^{-1})
\pmod{(\mathfrak{m}^{k+1})^{(d)}}.
\]
Hence also
$\phi(x)-\phi(y)\in (\mathfrak{m}^k)^{(d)} \smallsetminus (\mathfrak{m}^{k+1})^{(d)}$,
and (i) follows.

(ii)
Observe that (i) implies that
\begin{equation}
\label{product vs sum 2}
xy^{-1} \in \mathfrak{m}^n_\phi S
\
\Longleftrightarrow
\
\phi(x)-\phi(y) \in (\mathfrak{m}^{N+n})^{(d)},
\end{equation}
or what is the same,
\[
x \cdot \mathfrak{m}^n_\phi S = y \cdot \mathfrak{m}^n_\phi S 
\
\Longleftrightarrow
\
\phi(x)+(\mathfrak{m}^{N+n})^{(d)} = \phi(y)+(\mathfrak{m}^{N+n})^{(d)}.
\]
This proves (ii), since $\phi$ is a bijection.

(iii)
According to the definition of the metric associated to a fitration, it is clear that (ii) and (\ref{product vs sum 2}) together imply that $\phi$ is an isometry.
\end{proof}

Hausdorff dimension can be defined for any subset of a metric space, see \cite[Chapter 2]{Falconer} for its definition and main properties.
We need the following two lemmas about Hausdorff dimension in countably based profinite groups.

\begin{lemma}
\label{leqo}
Let $G$ be a countably based profinite group with filtration $\{G_n\}_{n\in\mathbb{N}}$.
Let $H$ be a closed subgroup of $G$, and let $U$ be a non-empty open subset of $H$.
Then
\[
\hDim_{\{G_n\}}(H) = \hDim_{\{G_n\}}(U).
\]
\end{lemma}

\begin{proof}
The proof is a straightforward consequence of \cite[Section 2.2]{Falconer}.
\end{proof}

Let $G$ be a countably based profinite group and let $S$ be an open subgroup of $G$.
If $\{S_n\}_{n\in\N}$ is a filtration of $S$, we can calculate the Hausdorff dimension of $X\subseteq S$ with respect to the metric induced by
$\{S_n\}_{n\in\N}$ in $S$ or in $G$, which we denote by $\hDim_{\{S_n\}}^S (X)$ and $\hDim_{\{S_n\}}^G (X)$, respectively.
Our next lemma shows that there is no need to make this distinction in the notation.

\begin{lemma}
\label{ino}
Let $G$ be a countably based profinite group and let $S$ be an open subgroup of $G$.
If $\{S_n\}_{n\in\mathbb{N}}$ is a filtration of $S$, then
\[
\hDim_{\{S_n\}}^G (X) = \hDim_{\{S_n\}}^S (X)
\]
for every $X\subseteq S$.
\end{lemma}

\begin{proof}
Let $d_S$ and $d_G$ be the metrics induced by $\{S_n\}_{n\in\N}$ in $S$ and $G$.
Then $d_S(x,y)=|G:S| \, d_G(x,y)$ for all $x,y\in S$, and the identity map from $(S,d_S)$ to $(S,d_G)$ is bi-Lipschitz.
Now the result follows from \cite[Corollary 2.4]{Falconer}.
\end{proof}

\section{Proof of the main theorem}

In this section we first prove that the Hausdorff dimension of a closed subgroup in an analytic profinite group with respect to a natural filtration is independent of the standard subgroup.
Then we prove the main theorem of our paper about Hausdorff dimension of analytic subgroups.

%In order to do this we start by giving sufficient conditions to get the same Hausdorff dimension with respect to different filtrations.

%We note that we are only going to need Proposition \ref{P1} to calculate the Hausdorff dimension of closed subgroups of $G$, and not of arbitrary subsets of $G$.
%In that case, an alternative proof could be given by using the formula in \eqref{A}.
%Since that proof is not significantly shorter than the proof given above, we have preferred to state the result in full generality.

%\begin{remark}
%\label{R1}
%Let $G$ be an $R$-analytic profinite group of dimension $d$, and let $S$ be an open standard subgroup of $G$, with corresponding chart $\phi$.
%Then, according to Lemma \ref{iso}, the natural filtration $\{\mathfrak{m}^n_\phi S\}_{n\in\N}$ always satisfies condition (ii) in Proposition \ref{P1}, with %$M=q^d$.
%\end{remark}

%A key consequence of Proposition \ref{P1} is that the Hausdorff dimension function does not depend on the choice of the standard subgroup of $G$.

%Actually, as we prove below, given two open standard subgroups of an $R$-analytic profinite group, the corresponding natural filtrations also satisfy condition (i) in
%Proposition \ref{P1}, and we have the following consequence.

\begin{theorem}
\label{std}
Let $G$ be an $R$-analytic profinite group, and let $(S,\phi)$ and $(T,\psi)$ be two open standard subgroups of $G$.
Then
\[
\hDim_{\{\mathfrak{m}^n_\phi S\}}(H) = \hDim_{\{\mathfrak{m}^n_\psi T\}}(H)
\]
for every closed subgroup $H$ of $G$.
\end{theorem}

\begin{proof}
Let us write $S_n$ and $T_n$ for $\mathfrak{m}^n_\phi S$ and $\mathfrak{m}_\psi^n T$, and $N(S)$ and $N(T)$ for the levels of $S$ and $T$.
We first show that there exist non-negative integers $a$ and $b$ such that 
\begin{equation}
\label{standard sandwich}
S_{n+a}\le T_n\le S_{n-b}
\end{equation}
for every $n\ge b$.
Since the charts $\phi$ and $\psi$ belong to the full atlas of $G$, the two functions $\psi\circ\phi^{-1}|_{\phi(S\cap T)}$ and $\phi\circ\psi^{-1}|_{\psi(S\cap T)}$ are analytic.
Since $\phi(S\cap T)$ is open in $\phi(S)$, it follows that $\psi\circ\phi^{-1}$ can be evaluated in $(\mathfrak{m}^{\ell})^{(d)}$ for some $\ell$.
By \cite[Lemma 6.45]{Dixon}, and observing that $\psi\circ\phi^{-1}(0)=0$, there exists a natural number $k$ such that
$$
\psi\circ\phi^{-1}((\mathfrak{m}^{n+k})^{(d)})\subseteq (\mathfrak{m}^n)^{(d)}
$$
for every $n\ge 0$.
This implies that $S_{n+a}\le T_n$ for $n\ge 0$, by choosing $a=\max\{k-N(S)+N(T),0\}$.
Arguing similarly with $\phi\circ\psi^{-1}$, we get (\ref{standard sandwich}).

Now, by \cite[Theorem 2.4]{YA}, we have
\begin{align*}
\hDim_{\{T_n\}}(H)
&=
\liminf_{n\rightarrow\infty} \, \frac{\log |H:H\cap T_n|}{\log |G:T_n|}
\\[5pt]
&\le
\liminf_{n\rightarrow\infty} \, \frac{\log |H:H\cap S_{n+a}|}{\log |G:S_{n+a}|-\log |T_n:S_{n+a}|}
\\[5pt]
&\le
\liminf_{n\rightarrow\infty} \, \frac{\log |H:H\cap S_{n+a}|}{\log |G:S_{n+a}|}\cdot \Bigg(1-\frac{\log |S_{n-b}:S_{n+a}|}{\log |G:S_{n+a}|} \Bigg)^{-1}
\\[5pt]
&=
\hDim_{\{S_n\}}(H).
\end{align*}
The last equality holds because, by (ii) of Lemma \ref{iso}, for $n$ large enough we have
$$
\log_q |S_{n-b}:S_{n+a}|=\sum_{i=N+n-b}^{N+n+a-1} \, P(i),
$$
which is polynomial in $n$ of degree $\deg P$; on the other hand, $\log_q |G:S_{n+a}|$ is asymptotically equivalent to
$$
\log_q |G:S| + \sum_{i=N}^{N+n+a-1} \, P(i) 
$$
and, by the Euler-MacLaurin formula, this sum is a polynomial in $n$ of degree $\deg P+1$.
By swapping $S$ and $T$, the result follows.
\end{proof}

\begin{definition}
Let $G$ be an $R$-analytic profinite group and let $H$ be a closed subgroup of $G$.
Then the \emph{standard Hausdorff dimension\/} of $H$, $\hDim_{\Std}(H)$, is the Hausdorff dimension of $H$ calculated with respect to the natural filtration induced by any given standard open subgroup of $G$.
\end{definition}

We need the following lemma before proving our main theorem.

\begin{lemma}
\label{R-modules}
Let $R$ be a pro-$p$ ring, and let $E$ be a vector subspace of dimension $e$ of $K^{(d)}$.
Then, for every $N\in\N$, the Hausdorff dimension of $(\mathfrak{m}^N)^{(d)}\cap E$ in $R^{(d)}$ with respect to the filtration $\{(\mathfrak{m}^n)^{(d)}\}_{n\in\N}$ is $e/d$.  
\end{lemma}

\begin{proof}
Let $\mathcal{B}=\{w_1,\ldots,w_e\}$ be a basis of $E$, and let $A$ be the $e\times d$ matrix over $K$ whose rows are the vectors in $\mathcal{B}$.
We may assume that $A$ is in reduced echelon form.
Set $W=\langle w_1,\ldots,w_e \rangle_R$.
Let $k$ be such that the product of all denominators of the entries of $A$ belongs to $\mathfrak{m}^k$.
One readily checks that
\begin{equation}
\label{sandwich W}
\mathfrak{m}^{n+k} \, W \subseteq (\mathfrak{m}^n)^{(d)}\cap E\subseteq \mathfrak{m}^n \, W.
\end{equation}
The required Hausdorff dimension is the limit inferior of the sequence
\[
c_n
=
\frac{\log |(\mathfrak{m}^N)^{(d)}\cap E:(\mathfrak{m}^n)^{(d)}\cap E|}{\log |R^{(d)}:(\mathfrak{m}^n)^{(d)}|},
\]
where $n\ge N$.
Now, from (\ref{sandwich W}), for $n\ge N+k$, we have
\[
|\mathfrak{m}^{N+k} \, W:\mathfrak{m}^n \, W|
\le
|(\mathfrak{m}^N)^{(d)}\cap E:(\mathfrak{m}^n)^{(d)}\cap E|
\le
|\mathfrak{m}^N \, W:\mathfrak{m}^{n+k} \, W|.
\]
Since $W$ is a free $R$-module of rank $e$, we have
$\dim_{R/\mathfrak{m}} (\mathfrak{m}^n\, W/\mathfrak{m}^{n+1}\, W)=e H(n)$, and consequently
\[
\frac{e (H(N+k)+\cdots+H(n-1))}{d (H(1)+\cdots+H(n-1))}
\le c_n
\le
\frac{e (H(N)+\cdots+H(n+k-1))}{d (H(1)+\cdots+H(n-1))}.
\]
Now, since $H(n)=P(n)$ for large enough $n$, we have
\[
\lim_{n\to\infty} \frac{H(N+k)+\cdots+H(n-1)}{H(1)+\cdots+H(n-1)}
=
\lim_{n\to\infty} \frac{P(N+k)+\cdots+P(n-1)}{P(1)+\cdots+P(n-1)}
=
1,
\]
since both sums in the last limit are polynomials in $n$ of degree $\deg P+1$ and with the same leading coefficient.
Arguing similarly with the upper bound for $c_n$ given above, we conclude that $\lim_{n\to\infty} \, c_n=e/d$, as desired.
\end{proof}

We are now ready to prove our main theorem.

\begin{proof}[Proof of the Main Theorem]
Let $\dim (G)=d$ and $\dim (H)=e$.
Since $H$ is an $R$-analytic submanifold of $G$, there exist an open subset $U$ of $H$ containing $1$, and a chart $(V,\phi)$ of $G$ such that 
$U\subseteq V$ and $\phi(U)=E\cap \phi(V)$, for some vector subspace $E$ of $K^{(d)}$ of dimension $e$.

 From the proof of \cite[Theorem 13.20]{Dixon}, there exist a natural number $N$ and an open subgroup $S$ of $G$ such that $S\subseteq V$ and $(S,\phi)$ is a standard subgroup of $G$ of level $N$.
Since $U\cap S\subseteq_o U\subseteq_o H$, we have
\[
\hDim_{\Std}(H)=\hDim_{\{\mathfrak{m}^n_{\phi} S\}}(U\cap S),
\]
by Lemma \ref{leqo}.
Since $\phi$ is an isometry between $S$ and $((\mathfrak{m}^N)^{(d)},+)$ by Lemma \ref{iso}, we get
\[
\hDim_{\{\mathfrak{m}^n_{\phi} S\}}(U\cap S)=\hDim_{\{(\mathfrak{m}^n)^{(d)}\}}(\phi(U\cap S)).
\]
Now since $\phi (U\cap S)= E\cap (\mathfrak{m}^N)^{(d)}$, Lemma \ref{R-modules} yields that
\[
\hDim_{\{(\mathfrak{m}^n)^{(d)}\}}(\phi(U\cap S))=\frac{e}{d},
\]
which completes the proof.
\end{proof}

\begin{corollary}\label{c}
Let $G$ be an $R$-analytic group of dimension $d$.
Then
\[
\Spec_R(G) \subseteq \Big\{ 0,\frac{1}{d},\ldots,\frac{d-1}{d},1 \Big\}.
\]
In particular, the $R$-analytic spectrum of $G$ is finite and consists of rational numbers.
\end{corollary}

\noindent
\textit{Acknowledgement\/}.
We thank the anonymous referee for helpful comments.

\end{document}